\documentclass{amsart}

\usepackage{graphicx}
\usepackage{amsmath,amsthm,amssymb, IMjournal}

\usepackage{color,comment,lipsum}

\begin{document}

\newtheorem{Theorem}{Theorem}
\newtheorem{Lemma}[Theorem]{Lemma}
\newtheorem{Proposition}[Theorem]{Proposition}
\newtheorem{Corollary}[Theorem]{Corollary}

\newtheorem{Definition}[Theorem]{Definition}
\newtheorem{Notation}[Theorem]{Notation}
\newtheorem{Example}[Theorem]{Example}
\newtheorem{Remark}[Theorem]{Remark}
\newtheorem{Remarks}[Theorem]{Remarks}

\newcommand\N{\mathbb{N}}
\newcommand\Z{\mathbb{Z}}
\newcommand\R{\mathbb{R}}

\title{A New Algorithm for Approximating the Least Concave Majorant}
\author{\| Martin | Franc\accent23u|, Prague,  \| Ron |Kerman|,  St. Catharines, \|  Gord |Sinnamon|,~London }

\abstract
The least concave majorant, $\hat F$, of a continuous function $F$ on a closed interval, $I$, 
is defined by 
\[
\hat F (x) = \inf \left\{ G(x): G \geq F, G \mbox{ concave}\right\},\; x \in I. 
\]
We present here an algorithm, in the spirit of the Jarvis March, to approximate the least concave majorant of a differentiable piecewise polynomial function of degree at most three on $I$. Given any function $F  \in \mathcal{C}^4(I)$, it can be well-approximated on $I$ by a clamped cubic spline $S$. We show that $\hat S$ is then a good approximation to $\hat F$.

We give two examples, one to illustrate, the other to apply our algorithm.

\endabstract



\thanks 
The first-named author was supported by the grant SVV-2016-260335 and by the grant  P201/13/14743S of the Grant Agency of the Czech Republic. NSERC support is gratefully acknowledged.
\endthanks

\section{Introduction}

Suppose $F$ is  a continuous function on the interval $I = [a,b]$. Denote by $\hat F$ the least concave majorant of $F$, namely, 
\[
\hat F (x) = \inf \left\{ G(x) :  G \geq F, G \mbox{ concave }\right\}, 
\] 
which can be shown to be given by 
\[
\hat F (x) = \sup \left\{ \frac{\beta - x} {\beta - \alpha} F (\alpha )  + \frac{ x - \alpha }{\beta - \alpha} F (\beta) : a \leq \alpha \leq  x \leq \beta \leq b \right\} , \; x \in I. 
\]
This concave function has application in such diverse areas as Mathematical Economics, Statistics, and Abstract Interpolation Theory. 
See, for example, \cite{Deb1976}, \cite{Car2002}, \cite{Pee1970}, \cite{BruKru1991}, \cite{MasSin2006} and \cite{KerMilSin2007}.
We observe that $\hat F$ is continuous on $I$ and it is differentiable there when $F$ is.  

Our aim in this paper is to give a new algorithm to approximate $\hat F$, together with an estimate of the error entailed. If $F$ is a continuous or, stronger yet, a differentiable piecewise  polynomial of degree at most three, then so is $\hat F$. If not, then $F$ may be approximated by a clamped cubic spline and the least concave majorant of the approximating function is seen to be a good approximation to $\hat F$. To estimate the error in Theorem~\ref{Theorem53} below we use a known result for the approximation error involving such cubic splines from \cite{HalMey1976}, together with a new result on $(\hat F)'$, which in \cite[p.70]{Lor1953} and  \cite{Halperin1953} is denoted by $(F')^{\circ}$ and is referred to as the level function of $F'$ in the unweighted case.
See the aforementioned Theorem \ref{Theorem53}.
\par
The simple structure of $\hat F$ will be the basis of our algorithm. Since $F$ and $\hat F$ are continuous, the zero set, $Z_F$, of $\hat F - F$ is closed; of course, $\hat F = F$ on $Z_F$. The connected components of $Z_F^c$ are intervals open in the relative topology of $I$ on which $\hat F$ is a strict linear majorant of $F$; indeed, if, for definiteness, the component interval with endpoints $\alpha$ and $\beta$ is a  subset of the interior of $I$, then   
\begin{equation}\label{1}
\hat F (\alpha) =F(\alpha), \; \hat F (\beta) = F (\beta),
\end{equation}
\begin{equation}\label{2}
F(x) < \hat F(x) = F(\alpha) + (x - \alpha ) \frac{F(\beta) - F (\alpha)}{\beta - \alpha}, \, \alpha < x   < \beta,
\end{equation}
and, if $F$ is differentiable on $I$, 
\begin{equation}\label{3}
(\hat F)' (\alpha) = F'(\alpha) =  \frac{F(\beta) - F (\alpha)}{\beta - \alpha} = F'(\beta) = (\hat F)'(\beta).
\end{equation}

Our task is thus to find the component intervals of $Z_F^c$. This will be done using a refinement of the Jarvis March  algorithm; see \cite{Jar1973}.  To begin, we determine the set of points, $D$, at which $F$ attains its maximum value, $M$, and then take $C=[c_1,c_2]$ to be the smallest closed interval containing $D$. Of course, in many cases $D$ consists of one point and $c_1 = c_2$. 
\par
It turns out that $\hat F$ increases to $M$ on $[a,c_1]$, is identically equal to $M$ on $C$, then decreases on $[c_2, b]$.
\par
To describe in general terms how the algorithm works we focus on $[a,c_1]$, $a < c_1$, and take $F$ to be a differentiable function which is piecewise cubic. As such, there is a partition, $P$,  of $[a,c_1]$ on each subinterval of which $F$ is a cubic polynomial. By refining the partition, if necessary, to include critical points and points of inflection of $F$, we may assume that this polynomial is  either strictly concave, linear or strictly convex and is either increasing or decreasing on its subinterval. It is the subintervals  where the associated cubic polynomial is increasing and strictly concave that are of interest. It is important to point out that for a piecewise cubic function, $Z_F^c$ has only finitely many components. 
\par
Now, $\hat F$ on a component of $Z_F^c$ may be thought of as a kind of linear bridge over a convex part of $F$. With this in mind,  we call an interval, say $J = (\alpha, \beta)$, a bridge interval if,   on it, $F$ satisfies  
\begin{equation}\label{bridge2}
F(x) < F(\alpha) + (x - \alpha ) \frac{F(\beta) - F (\alpha)}{\beta - \alpha}, \, \alpha < x   < \beta,
\end{equation}
and
\begin{equation}\label{bridge3}
 F'(\alpha) =  \frac{F(\beta) - F (\alpha)}{\beta - \alpha} = F'(\beta).
\end{equation} 
We include endpoints of $I$ as possible endpoints of bridge intervals. In such case, the corresponding part of (\ref{bridge3}) is omitted. 
An illustrating example of bridge intervals and least concave majorant of a function can be found in figure~\ref{figure6}. It might be helpful to reader to check demonstrative Example 1. in section~\ref{section7} while reading the formal description of algorithm. The algorithm is there applied to a particular spline. 
\par
Proceeding systematically from $c_1$ to $a$ (the procedure from $c_2$ to $b$ is similar) our algorithm determines, in a finite number of steps, a finite number of pairwise disjoint bridge intervals with endpoints in the intervals of increasing strict concavity referred to in the above paragraph.  
The desired components are among these bridge intervals.

 The technical details of all this are elaborated in Section \ref{section2}. Proofs of results stated in that section are proved in the next one and the algorithm itself is justified in the one following that. 
 Remarks on the implementation of the procedure are made in Section \ref{section5}. Section \ref{section6} has estimates of the error incurred when approximating an absolutely continuous function by a clamped cubic spline, while in the final section two examples are given. 

 \section{The algorithm}\label{section2}
    
  In this section we describe  our algorithm in more detail. This will require us to first state some lemmas whose proof will be given in the next section.
  \par
 Suppose that $F$ a is continuous function on some interval $I = [a,b]$ and let $\hat F, Z_F^c, M, D$ and $C= [c_1, c_2]$ be as in the introduction. 
   
\begin{Lemma}\label{SeptemberLemma1}
If $F$ is a continuous function on $I$, then the least concave majorant, $\hat F$, of $F$ on $I = [a,b]$ is continuous on $I$, with $\hat F(a) = F(a)$ and $\hat F(b) = F(b)$. Moreover, on each component interval, $J$, of $Z_F^c$, with endpoints $\alpha$ and $\beta$, $\hat F$ is the linear function, $l$, interpolating $F$ at the points $\alpha$ and $\beta$. 
\end{Lemma}
\begin{Lemma}\label{SeptemberLemma2}
Suppose $F$ is differentiable on $(a,b)$ and $(\alpha,\beta)$ is a component of $Z_F^c$. Then $\hat F$ is differentiable on $(a,b)$, $(\hat F)'(x)=F'(x)$ for $x\in (a,b)\cap Z_F$, and $(\hat F)'(x)=\frac{F(\beta)-F(\alpha)}{\beta-\alpha}$ for $x\in [\alpha,\beta]$. In particular, 
$
F'(x)=(\hat F)'(x)=\frac{F(\beta)-F(\alpha)}{\beta-\alpha}
$
if $x=\alpha\in (a,b)$ or $x=\beta\in (a,b)$. Moreover, if $F'$ is continuous on $(a,b)$, then so is $(\hat F)'$.  

\end{Lemma}

  \begin{Lemma}\label{OldLemma1} 
 Let $F$ be a continuous function on $I$, then  $\hat F \equiv M$ on $C$. Moreover, $\hat F$ is strictly increasing on $(a,c_1)$ and strictly decreasing on $(c_2, b)$.
  \end{Lemma}  
  \begin{Lemma}\label{observation}
  Let $F$ be  a continuous function, suppose $C=[c_1,c_2]$ is as in the introduction, and suppose $x,y,z \in (a,b)$ such that $F$ is strictly convex on $(x,z)$ and  $y \in (x,z)$. Then  $F(y) \neq \hat  F (y)$. 
  
  Suppose $F$ is differentiable as well. If $y \in (a, c_1)$ and $F'(y) \leq 0 $ then $F(y) \neq \hat F(y)$. Analogously, if $y \in (c_2,b)$ and $F'(y) \geq 0$ then $F(y) \neq \hat  F(y)$.  
  \end{Lemma}

  \begin{Lemma}\label{OldLemma2}
 Let $F$ be a continuous function.If $J = (\alpha, \beta)$ is a component interval of $Z_F^c$ then either $J \subset (a,c_1)$, $J \subset (c_1, c_2)$ or $J \subset (c_2, b)$. 
  \end{Lemma}
 Suppose that $F$ is piecewise cubic and differentiable on $I$, and suppose $J \subset [a,c_1]$.  
 Denote by $P$ the closed intervals determined by the partition of $[a,c_1]$ inherited from the piecewise cubic structure of $F$,  together with any critical points and points of inflection of $F$ in $[a,c_1]$.
  \par
  \begin{Lemma}\label{OldLemma3}
Suppose that $F$ is piecewise cubic and differentiable on $I$. Let $J = (\alpha, \beta) \subset [a,c_1)$ be a component interval of $Z_F^c$. Then, either $\alpha = a$ or there is an interval $K = [k_1, k_2]$ in $P$ containing $\alpha$ on which $F$ is strictly concave and increasing. Similarly, either $\beta = c_1$ or there is an interval $L = [l_1, l_2]$ in $P$ containing $\beta$  on which $F$ is strictly concave and increasing. Moreover, $K \neq L$.  
  \end{Lemma} 
  \par
  Leaving aside the case $c_1 = c_2 = b$ our goal is to select the components of $Z_F^c$ from among the bridge intervals of the form $[a,b_1)$ or $(a_1,b_1)$, $a_1 > a$, such that $a_1$ and $b_1$ lie in distinct intervals in $\mathcal{P}$ with disjoint interiors  on which intervals $F$ is strictly concave and increasing.
  \par
    Let $\mathcal{P}$ be the collection of intervals in $P$ where $F$ is strictly concave and increasing.
    \par
    Given a pair of intervals in $\mathcal{P}$ that could have the endpoints of a bridge interval in them, one determines those endpoints, if they exist, by the study of a certain sextic polynomial equation. The details of the most complicated case are described in the following lemma.
    
\begin{Lemma}\label{OldLemma4}    
Let $L = [l_1, l_2]$ and $R= [r_1, r_2]$ belong to $\mathcal{P}$ with $l_2 \leq r_1$. Suppose 
\[
F(x) = \begin{cases}
P_L (x) &= A x^3 + B x^2 + C x + D \mbox{ on } L\\
P_R (x) &= W x^3 + X x^2 + Y x + Z \mbox{ on } R,\\
\end{cases}
\] 
with $AW \neq 0$. Assume 
\[
J = P_L' (L) \cap P_R' (R) \neq \emptyset.
\]
Then, if there is a bridge interval $I_1 = (a_1, b_1)$ with $a_1 \in L$ and $b_1 \in R$, this bridge interval is such that 
\begin{equation}
a_1 = (P_L')^{-1}(y_0) \mbox{ and } b_1 =  (P_R')^{-1}(y_0), 
\end{equation}
where $y_0$ is a point in $J$ satisfying the sextic equation
\[
(\mu^2_1 - \mu^2_2 -  \mu^2_3 \delta )^2  -  4 \mu_2^2 \mu^2_3 \gamma \delta = 0,
\]
in which
\[
\gamma =  3 A y + B^2 - 3 AC ,\; \delta = 3 W y + X^2 - 3 W Y, \; \mu_2 = \frac{- 2 \gamma} {27 A^2 }, \mu_3 = \frac{2 \delta} {27 W^2  } 
\]
and 
\[
\mu_1 = \frac{1}{3} \left( \frac{X}{W} - \frac{B}{A} \right)y + \left( Z + \frac{2 X^3}{27 W^2}  - \frac{Y X}{3 W} \right)  - \left( D + \frac{2 B^3}{27A^2} - \frac{BC}{3A} \right).
\]
\end{Lemma}
 The verification that a given interval $J  = (\alpha,\beta)\subset (a,c_1)$ satisfies condition (\ref{bridge2}) can be achieved using the following criterion: 
Assume that $\alpha \in L = [l_1, l_2] \in \mathcal{P}$, $\beta \in R = [r_1, r_2] \in \mathcal{P}$, $l_2 < r_1$,  and that $l$ is a linear function interpolating $F$ on $J$. 
Then $J$ satisfies (\ref{bridge2}) if, for every $K = [k_1,k_2]$ in $P$, with $K \subset [l_2 ,r_1]$,
\[
l(k_1) - F(k_1) > 0 \mbox{ and } l(k_2) - F(k_2)>0,
\]     
and, in addition, if $K \in \mathcal{P}$, then
\[
l(\varrho) - F(\varrho) >0
\] 
for any root, $\varrho$, in $K$ of the quadratic    
\[
F'(x)  = \frac{F(\beta) - F(\alpha)}{\beta - \alpha}.
\]
Obvious modifications of the above must also hold for $[a_1, l_2]$ and $[r_1,b_2]$. This criterion can be proved using elementary calculus.

We are now able to describe an iterative procedure that selects the component intervals of $Z_F^c$ from a class of bridge intervals. We will focus our description on the case of finding all component intervals contained in $(a,c_1)$ as the case in which the component intervals are contained in $(c_2, b)$ is analogous while the component intervals in $(c_1, c_2)$ are determined trivially by Lemma~\ref{OldLemma1} . 

If $a = c_1$, then there is no such component interval. In the following, we exclude, at first, the  case $c_1 = c_2 = b$, so that $c_1 < b$. 
Set $\mathcal{P}_0 = \mathcal{P}$. 

 We claim that $\mathcal{P}_0$ cannot be empty. As a consequence of Lemma~\ref{OldLemma2}  we have that $\hat F (c_1) = F (c_1)$, since $c_1$ cannot be in interior of any component interval. The point $c_1$ is a local maximum of $F$. The choice of $P$ ensures that there is  
an interval $(x,c_1)$ such that $F$ is increasing and concave on it, hence $\mathcal{P}_0$ must contain at least one interval. 

Assume $\mathcal{P}_0$ has exactly one interval.  The fact that $c_1$ is a local maximum of $F$ ensures that this interval is of a form $[x,c_1)$.  Suppose now that $x = a$ then $F= \hat F$ on $[a,c_1]$, since the function
\[
m(t) = \begin{cases}
F(t), \; t \in [a,c_1], \\
M, t \in (c_1, b],
\end{cases}
\]
is a concave majorant of $F$.  (It is a concave function extended linearly  with slope that of the tangent line at the endpoint.) 
 
 Suppose now that $x \neq a$.  We have $F \neq \hat F$ on $(a,x)$ - if there were $y \in (a,x)$ such that $F(y) = \hat F(y)$, then $F$ would have to be increasing and strictly concave on some neighbourhood by Lemma~\ref{observation}  and Lemma~\ref{OldLemma3}. This is a contradiction to the assumption that $[x,c_1]$ is the only interval in $\mathcal{P}_0$. Since $F \neq \hat F$ on $(a,x)$ there must be a component interval containing $(a,x)$. On the other hand, Lemma~\ref{OldLemma1}  implies that $F(c_1) = \hat F(c_1)$, hence this component interval must be a subset of $(a,c_1)$. 

The desired component interval is of a form $(a, \beta)$, $\beta \in [x, c_1)$. If we choose $\beta$ to be the unique solution to the equation
\[
F'(\beta) = \frac{F(\beta) - F(a)}{\beta -a},
\]
then the interval $(a,\beta)$ will be the component interval, since it is the only interval which satisfies the necessary conditions (\ref{3}). 
\par
Suppose next that $\mathcal{P}_0$ has at least two intervals and take $R = [r_1 ,r_2]$ to be that interval in $\mathcal{P}_0$  closest to $c_1$. 
\par 
We seek first  a component interval of the form $(a,r)$, $r \in R$, as if $R$ were the only interval in $\mathcal{P}_0$. If no such interval exists, let $L = [l_1, l_2]$ be the interval in $\mathcal{P}_0$ closest to $a$, then use  Lemma \ref{OldLemma4} to test for a bridge interval $W = (w_1, w_2)$ with $w_1 \in L$ and $w_2 \in R$. 

It is important to point out that Lemma~\ref{OldLemma4} only places a restriction on bridge intervals, it does not guarantee them. Once the sextic is solved, condition (\ref{2}) must still be verified for the proposed bridge interval. This means iterating through each partition subinterval contained in the proposed bridge interval and solving a maximum problem to verify that $F$ lies underneath the proposed linear $\hat F$. 

In a true Jarvis March points, rather than intervals, are ordered according to the angle of a tangent line. In the case of intervals associated to piecewise {\it cubic} functions such an ordering is computationally expensive. 

 Should there be no such $W$ carry out the same test  on the interval in $\mathcal{P}_0$ closest to the right of $L$, if one exists.   
\par 
If, in moving systematically to the right in this way, we find  no $W$, we discard $R$ from $\mathcal{P}_0$ to get $\mathcal{P}_1$ and repeat the above procedure.
\par 
 If, on the contrary, we find such a $W$, it  will be a component interval. Say $w_1 \in N=[n_1, n_2]$, $N \in \mathcal{P}_0$.
\par
We next form $\mathcal{P}_1$ by discarding from $\mathcal{P}_0$ all intervals to the right of point $w_1$, for example $R$, and, in addition, replace $N$  by the interval $[n_1,w_1]$ (if $n_1 < w_1$, otherwise just discard $N$). We then carry out the above-described procedure with $\mathcal{P}_1$, if $\mathcal{P}_1 \neq \emptyset$.
\par
Continuing in this way we see that $\mathcal{P}_{n+1}$ has at least one less interval than $\mathcal{P}_n$, so the algorithm terminates after a finite number of steps. 
\par
 Finally, in the case $c_1 = c_2 = b$ there may be a component interval of $Z_F^c$ of the form $(r,b)$, $r \in [a,b)$. This may be found in a similar way as those of the form $(a,r)$.       
 
 \begin{Remark}
 We now comment briefly on how one can modify  our algorithm to deal with piecewise cubic functions that are only continuous. In this case the notion of a bridge interval has to be changed since function $F$ might not be differentiable at the endpoint of a component intervals of $Z_F^c$ and hence that end point needn't belong to an interval of strict concavity. 
 Accordingly, we say that $(\alpha, \beta)$ is  a bridge interval if conditions (\ref{bridge2}) and (\ref{bridge3}) hold and, in addition, 
 \[
 F'(\alpha - ) \geq \frac{ F (\beta) - F (\alpha)}{\beta - \alpha} \geq F'(\alpha+) \mbox{ and } F'(\beta -) \geq \frac{ F (\beta) - F (\alpha)}{\beta - \alpha}  \geq F'(\beta +).
 \]
 \par
 Again, Lemma \ref{OldLemma3} must be modified to compensate for the $F$ need not be differentiable. To do this we allow for {\it three} possibilities, namely, $\alpha = a$, $\alpha$ is contained in interval of strict concavity of $F$ {\it or} $\alpha$ is one of the points at which $F'(\alpha - ) > F'(\alpha +)$; a similar change must be made at the $\beta$. These changes necessitate our including all points of discontinuity of $F'$ as degenerate intervals in $\mathcal{P}$. 
 \par
 The iterations of our algorithm proceed much as in the differentiable case, with the difference that when some point, say $x$, is selected from $\mathcal{P}_i$ we must check if $(\alpha ,x)$ (or $(\beta, x)$ ) is a bridge interval in the new sense.  This can be done in a manner similar to the one we described for determining if $(\alpha,\beta)$ is a bridge interval in the old sense. 
 \end{Remark}

 \section{Proof of Lemmas 1-7}\label{section3}
 
 \begin{proof}[Proof of Lemma \ref{SeptemberLemma1}]

 Since $\hat F$ is concave it is continuous on the interior of $I$. This continuity ensures that, for all $\varepsilon >0$, there exists a slope $m$ such that the graph of $F$ lies under the line
 \[
 l_a(x) = F(a) + m (x- a ) + \varepsilon.
 \]
 But then $l_a$ would be a concave majorant of $F$, so
 \[
 F(x) \leq \hat F(x) \leq l_a(x), \; x \in I.
 \]
As $\varepsilon >0$ is arbitrary, $\hat F$ is continuous at $a$, with $\hat F(a) = F(a)$. 
A similar argument shows $\hat F$ is continuous at $b$, with $\hat F(b) = F(b)$.

Let $J$ and $l$ be as in the statement of Lemma~\ref{SeptemberLemma1} and suppose $y$ is a point at which $F - l$ achieves its maximum value on $I$. Since $F$ lies below the line $l + F(y) - l(y)$, so does $\hat F$. In particular, $\hat F (y) \leq F(y)$, so $\hat F(y) = F(y)$ and hence $y \notin J^\circ$. But, $\hat F(\alpha) = F(\alpha)$ and $\hat F(\beta) = F(\beta)$, so, by concavity, $\hat F$ lies above $l$ on $J$ and below $l$ off $J^\circ$. Thus,
\[
F(y) - l(y) \leq \hat F(y) - l(y) \leq 0,
\]
whence 
\[
F \leq l + F(y) - l(y) \leq l.
\]
This means $\hat F$ lies below $l$ on $J$. It follows that $\hat F = l$ on $J$. 
 \end{proof}
 \begin{proof}[Proof of Lemma \ref{SeptemberLemma2}]
 If $x \in (a,b)\cap Z_F$ then $\hat F (x) = F(x)$. Since $\hat F$ is a concave majorant of $F$, for any $w$ and $y$ satisfying $a <w < x < y < b$, we have 
 \[
 \frac{F(y) - F(x) }{y-x} \leq \frac{\hat F(y) - \hat F(x)}{y-x} \leq \frac{\hat F(x) - \hat F(w) }{x-w} \leq \frac{F(x) - F(w)}{x-w}  =\frac { F(w) -  F(x) }{w-x}.
 \]
Since $F$ is differentiable at $x$, the squeeze theorem shows that $(\hat F)'(x)$ exists and equals $F'(x)$. 

Lemma~\ref{SeptemberLemma1} shows that, on $(\alpha,\beta)$, $\hat F$ is a line with slope $\frac{F(\beta)-F(\alpha)}{\beta-\alpha}$. So it is differentiable on $(\alpha, \beta)$ and has one-sided derivatives at the points $\alpha$ and $\beta$. If $\alpha$ or $\beta$ is in $(a,b)\cap Z_F$ the derivative of $\hat F$ exists there and, of course, coincides with its one-sided derivative. If $\alpha=a$ or $\beta=b$, the endpoints of the domain of $\hat F$, then $(\hat F)'$ is a necessarily just a one-sided derivative. We conclude that $(\hat F)'=\frac{F(\beta)-F(\alpha)}{\beta-\alpha}$ on the closed interval $[\alpha, \beta]$.

Evidently, $(\hat F)'$ is continuous at each $x\in X_F^c$. Suppose $F'$ is continuous at $x\in(a,b)\cap Z_F$. If $a < w < x <y < b$ then any component of $Z_F^c$ that intersects $(w,y)$ has at least one endpoint in $(w,y)$. It follows that $(\hat F)'(w,y) \subset F'(w,y)$. Since $F'$ is continuous at $x$, so is $(\hat F)'$.

 \end{proof}
 
 \begin{proof}[Proof of Lemma \ref{OldLemma1}]
 To verify the first statement, one need only observe that between two points in $D$ (at which $F = M$) $\hat F = M$. 
 
 The second statement follows from a simple contradiction argument: Assume that there are $x_1, x_2 \in (a, c_1)$, $x_1 < x_2$, such that  $\hat F (x_1) \geq \hat F(x_2)$. Then $\hat F (x_1) < \hat F (c_1)$ implies that 
 \[
 \hat F(x_2) < \hat F(x_1) \frac{c_1 - x_2}{c_1 - x_1} +  \hat F (c_1)  \left(  1 - \frac{ c_1 - x_2}{c_1 - x_1}\right).
 \]  
But this contradicts the concavity of $\hat F$. Consequently, we have $\hat F (x_1) < \hat F (x_2)$. An analogous argument shows that $\hat F$ is strictly decreasing on $(c_2, b)$.  
 \end{proof}
 
 \begin{proof}[Proof of Lemma~\ref{observation}]
 The second part follows from Lemma~\ref{OldLemma1}, as $\hat F$ is strictly increasing on $(a,c_1)$ and strictly decreasing on $(c_2,b)$. This leads to contradiction as if $F(y) = \hat F (y)$ then $F'(y) = (\hat F)' (y)$ by Lemma~\ref{SeptemberLemma2} if $y$  is an isolated point of $Z_F$ 
 and trivially otherwise. 
 
To prove the first part: suppose for contradiction that $\hat F (y) = F (y)$. Then
 \[
 \hat F(y) = F (y) \leq  F(x) \frac{y-x}{z-x} + F(z) \frac{z-y}{z-x} \leq \hat F (x) \frac{y-x}{z-x} + \hat F(z) \frac{z-y}{z-x}, 
 \]
 which is in contradiction with the strict concavity of $\hat F$. 
 \end{proof}

 \begin{proof}[Proof of Lemma \ref{OldLemma2}]
 For  $x$ in bridge interval $J = (\alpha, \beta)$, condition(\ref{bridge2}) yields
 \begin{eqnarray*} 
 F(x) &\leq&  F (\alpha)\frac{\beta - x }{\beta - \alpha}  + F (\beta )  \frac{x - \alpha }{\beta - \alpha}\\
 &\leq& M,
 \end{eqnarray*}
  with equality only with $F (\alpha) = F(\beta) = M$. Thus, $J$ intersects $C$ only if both endpoints are contained in $C$. The conclusion follows. 
 \end{proof} 
 \begin{proof}[Proof of Lemma \ref{OldLemma3}]
 When $a =\alpha$ or $b = c_1 = \beta$ there is nothing to prove. Assume first, then, that $\alpha > a$ and choose  $K = [k_1, k_2] \in P$ such that  $\alpha\in[k_1,k_2)$. For any $x\in J\cap(\alpha,k_2)$, Lemmas 2 and 3 combine to give, 
\[
F(x)<\hat F(x)=F(\alpha)+(x-\alpha)F'(\alpha).
\]
Since $F$ lies below its tangent line, it is neither linear nor strictly convex on $[k_1,k_2]$.

Thus, $F$ must be strictly concave on $K$.  Lemma~\ref{OldLemma1}  implies that $\hat F$ is strictly increasing on $(a,c_1)$, hence $(\hat F)'(\alpha) > 0$.  Lemma~\ref{SeptemberLemma2} yields that $(\hat F)'$ exists and $(\hat F)'(\alpha) =  F' (\alpha)$.  The choice of $P$ ensures that $F$ is monotone on $K$. Hence $F$ is increasing on $K$. 
\par
A similar argument yields $F$ strictly concave and increasing on $L = [l_1,l_2]$ when $\beta < c_1$.   
 \end{proof}
 
 \begin{proof}[Proof of Lemma \ref{OldLemma4}]
 Since $F'$ is decreasing on $L$ and $R$, $J = [c,d]$, with $c = \max \left[F'(l_2), F'(r_2) \right]$ and $d = \min \left[F'(l_1), F'(r_1) \right]$.
 \par 
 Now, 
 \[
 F'(x) = 
 \begin{cases}
P_L' (x) = 3 A x^2 + 2 B x + C \mbox{ on } L,\\
P_R' (x) = 3 W x^2 + 2 X  x + Y \mbox{ on } R,\\
 \end{cases}
 \]
 with $F'$ decreasing on both intervals. So, the unique root, $a(y)$ and $b(y)$, of 
 \[
 P_L' (a(y)) = y \mbox{ and } P_R' (b(y)) = y, \, y \in J,
 \]
 can be obtained from the formulas
 \[
 a(y) = -  \frac{1}{3A} [B \pm \sqrt{3Ay + B^2 - 3AC}]
 \]
 and 
 \[
 b(y) = -\frac{1}{3W} [X \pm \sqrt{3Wy + X^2 - 3WY}].
 \]
 We now seek $y \in J$ so that 
 \[
 \frac{F(b(y)) - F(a(y))}{b(y) - a (y) } = y
 \]
 or 
 \begin{equation}\label{(5)}
 F(b(y)) - y b(y) - (F(a(y)) - y a (y) ) = 0.
 \end{equation}
 \begin{figure}
\includegraphics[scale = 0.28]{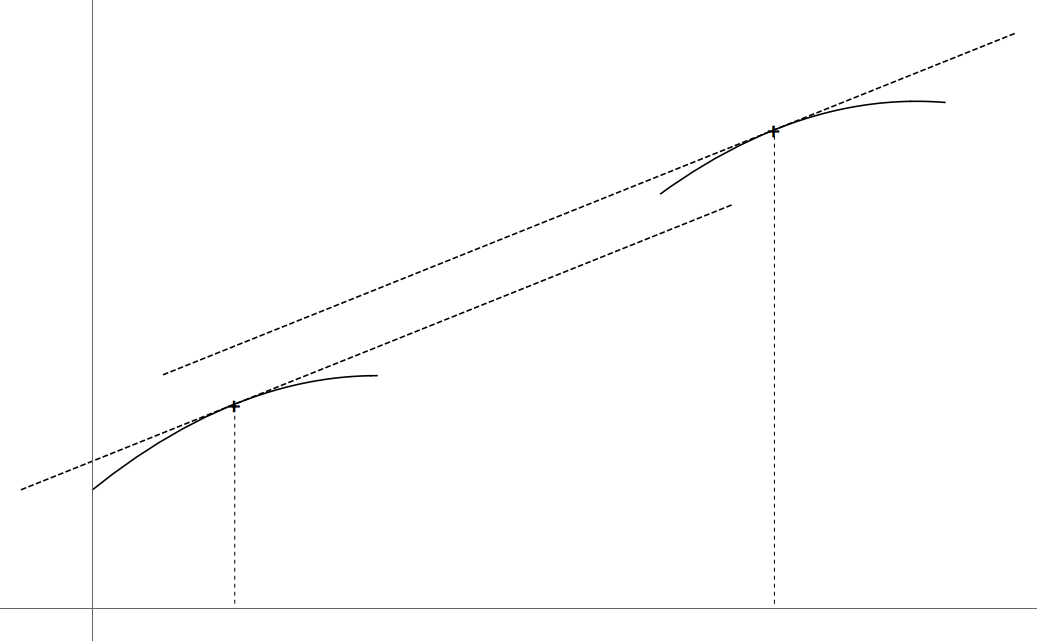}
\put(-230,0){$a(y)$}
\put(-80,0){$b(y)$}
\put(-187,65){$P_L$}
\put(-30,140){$P_R$}
\caption{For each $y \in (P_L)'(L) \cap (P_R)'(R)$ there exist exactly one $a(y) \in L$ and $b(y) \in R$ such that $P_L' (a(y)) = R_L'(b(y)) = y$.}
\label{figureSexticWrong}
\end{figure}
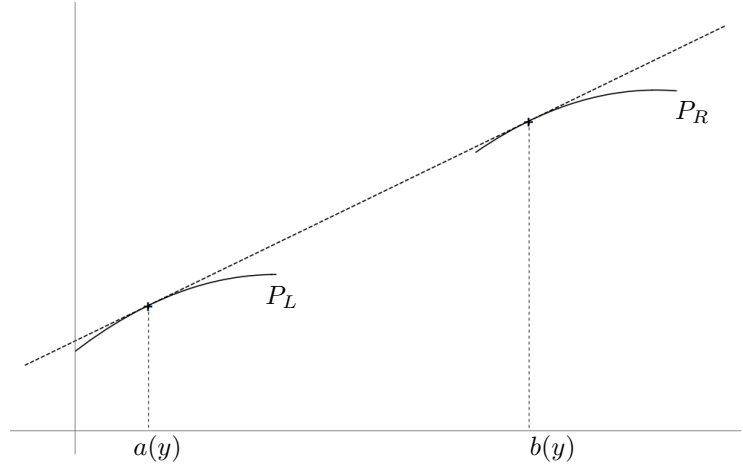
\begin{figure}
\includegraphics[scale = 0.28]{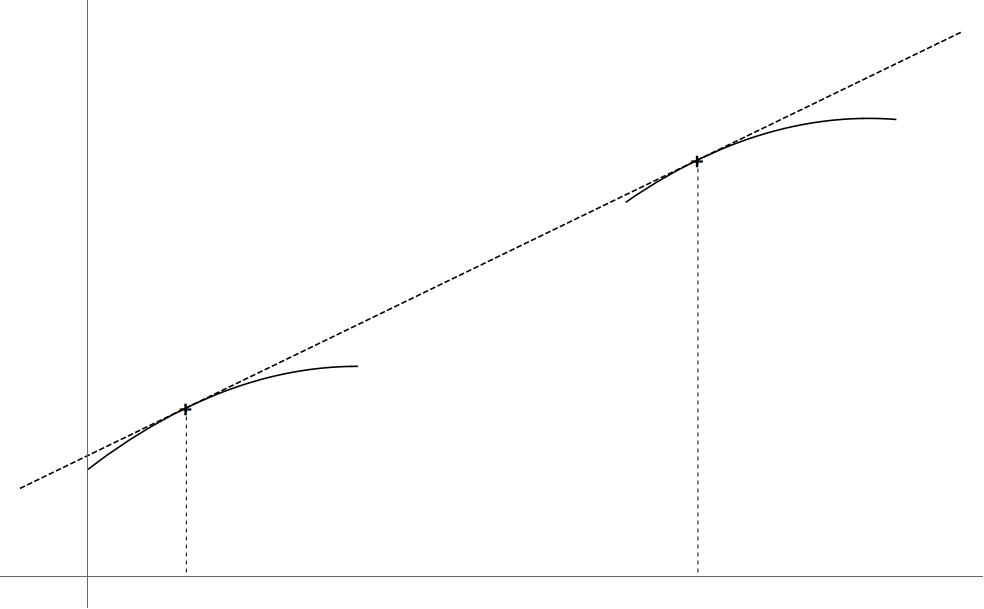}
\put(-230,0){$a(y)$}
\put(-80,0){$b(y)$}
\put(-180,57){$P_L$}
\put(-25,127){$P_R$}
\caption{There is a $y_0 \in (P_L)'(L) \cap (P_R)'(R)$ so that the corresponding $a(y_0)$ and $b(y_0)$ referred to in the caption of Figure~\ref{figureSexticWrong} satisfy 
$
y_0 = \frac{F(b(y_0)) - F(a(y_0))} {b(y_0) - a(y_0)} = \frac{P_R(b(y_0)) - P_L(a(y_0))}{b(y_0) - a(y_0)},
$ 
whence 
$
P_L'(a(y_0)) = P_R'(b(y_0)) = y_0 = \frac{F(b(y_0)) - F(a(y_0))}{b(y_0) - a(y_0)}.
$ }
\label{figureSexticRight}
\end{figure}
 Figures~\ref{figureSexticWrong} and Figures~\ref{figureSexticRight} below illustrate the geometric meaning of equation (\ref{(5)}). 
 Letting 
 \[
 \gamma(y) = 3Ay + B^2 - 3AC \mbox{  and } \delta(y) = 3Wy + X^2 - 3WY,
 \]
 the equation (\ref{(5)}) is equivalent to 
 \begin{equation}\label{6}
 \mu_1 + \mu_2 \sqrt{\gamma} + \mu_3 \sqrt{\delta} = 0,
 \end{equation}
 with $\mu_1, \mu_2, \mu_3$ linear functions of $y$, namely,
 \[
 \mu_2 (y) = - \frac{2 \gamma}{27 A^2 }, \; \mu_3 = \frac{2\delta} {27 W^2}
 \]
 and 
 \[
 \mu_1 (y) = \frac{1}{3} \left(\frac{X}{W} - \frac{B}{A} \right) y + \left( Z + \frac{2X^2}{27  W^2} - \frac{YX}{3W}  \right) - \left( D + \frac{2 B^2}{27 A^2} - \frac{CB}{3A}  \right).
 \]
 We claim the solution of (\ref{(5)}) is a root of the sextic polynomial equation
 \begin{equation}\label{7}
 (\mu_1^2 - \mu_2^2 \gamma - \mu_3^2 \delta )^2 - 4 \mu_2^2 \mu_3^2 \gamma \delta = 0.
 \end{equation}
 Indeed, isolating $\mu_1$ in (\ref{6}), then squaring both sides gives 
 \begin{equation}\label{8}
 \mu_1^2 = \mu_2^2 \gamma + \mu_3^2 \delta + 2 \mu_2 \mu_3 \sqrt{\gamma} \sqrt{\delta}.
 \end{equation}
 Isolating the term in (\ref{8}) with the square roots and squaring both sides yields  (\ref{7}).
 \end{proof}
 
 The following remark is given to make the appearance of the sextic equation seem more natural.
 \begin{Remark}
 Suppose, for definiteness, the $a(y)$ and $b(y)$ referred to in the proof of Lemma~\ref{OldLemma4} are given by 
 \[
 a(y) = \frac{-B}{3A} + \frac{1}{3A} \sqrt{3Ay + B^2 - 3AC} \mbox{ and } b(y) = \frac{-X}{3W} + \frac{1}{3W} \sqrt{3WY + X^2 - 3WY}.
 \]
 Then, equation (\ref{(5)}) can be written
 \begin{multline*}
 P_R \left( \frac{-X}{3W} + \frac{1}{3W} \sqrt{3WY + X^2 - 3WY} \right) - P_L \left( \frac{-B}{3A} + \frac{1}{3A} \sqrt{3Ay + B^2 - 3AC} \right)  \\
 = y \left( \frac{-X}{3W} + \frac{B}{3A} + \frac{1}{3W} \sqrt{3WY + X^2 - 3WY} - \frac{1}{3A} \sqrt{3Ay + B^2 - 3AC}  \right). 
 \end{multline*}
 In our original proof of Lemma~\ref{OldLemma4} we rearranged the terms in this version of  (\ref{(5)}), then squared both sides. We repeated this procedure a few times to get rid of the square roots and so arrive arrive at the sextic equation (\ref{7}).
 \end{Remark}

\section{Justification of the algorithm}\label{section4}
The purpose of this section is to prove 
\begin{Theorem}\label{OldTheorem5}
Let $F$ be differentiable piecewise cubic function. Then  the bridge intervals coming out of the algorithm are precisely the component intervals of $Z_f^c$. 
\end{Theorem}

For simplicity, we consider only the components in $[a,c_1)$. We begin with the preparatory
 \begin{Lemma}\label{OldLemma6}
Suppose that $F$ is absolutely continuous on $I_0 = [a,b]$.   Let $I = (a_1, b_1 )$ be a bridge interval with right hand endpoint in an interval $R$ on which $F$ is strictly concave and increasing. If $J = (a_2, b_2)$ is another bridge interval such that $I \cap J \neq \emptyset$, $b_2 \in R$ and $b_1 < b_2$, then $a_2 < a_1$. 
 \end{Lemma}
\begin{proof}
Let 
\[
l_I (x) = F(a_1) + (x-a_1) \frac{F(b_1) - F(a_1)}{b_1 - a_1}
\]
and, similarly,
\[
l_J (x) = F(a_2) + (x-a_2) \frac{F(b_2) - F(a_2)}{b_2 - a_2}.
\] 
Assume, if possible, $a_1 < a_2$. Then, $a_2 < b_1$, otherwise $I\cap J = \emptyset$. 
So,
\begin{equation}\label{lemma6contradiction}
l_J(a_2) = F(a_2) < l_I(a_2),
\end{equation}
since $I$ is a bridge interval.  The latter also implies 
\[
F(b_2) = l_I (a_2) + (b_1 - a_2) F'(b_1) + \int_{b_1}^{b_2} F'(t) \, d t;  
\]
further, $J$ being a bridge interval, we have 
\[
F(b_2) = l_J (a_2) + (b_2 - a_2)F'(b_2).
\] 
Therefore, 
\[
0 =   l_J (a_2) - l_I (a_2) + (b_2 - a_2)F'(b_2) -  (b_1 - a_2) F'(b_1) - \int_{b_1}^{b_2} F'(t) \, d t.
\]
The strict concavity of $F$ on $R$ ensures that $F'(t) > F'(b_2)$ for $t \in R$, $t < b_2$.  Thus
\begin{eqnarray*}
 l_I (a_2) - l_J (a_2) &=& (b_2 - a_2) F'(b_2) - (b_1 - a_2) F'(b_1 )- \int_{b_1}^{b_2} F'(t) \, d t \\
&<& (b_2 - a_2) F'(b_2) - (b_1 - a_2) F'(b_2 )- \int_{b_1}^{b_2} F'(t) \, d t \\
&=& (b_2 - b_1) F'(b_2) - \int_{b_1}^{b_2} F'(t) \, d t < 0.
\end{eqnarray*}
Consequently,
\[
l_I(a_2) - l_J (a_2) < 0, 
\]
thereby contradicting (\ref{lemma6contradiction}). 
\end{proof}
\begin{proof}[Proof of Theorem \ref{OldTheorem5}.] 
 As a consequence of Lemma~\ref{OldLemma2} one gets that the  component intervals are split into three groups: component intervals contained in $[a,c_1]$, $[c_1,c_2]$ and component intervals which are subsets of $[c_2,b]$. We begin by observing that component intervals of $Z_F^c$ in $[a,c_1]$ are the maximal bridge intervals there.
\par
To the end of showing every bridge interval coming out of the algorithm is a component interval of $Z_F^c$, fix an iteration, say the $n$-th, of the procedure. Let $R=[r_1, r_2]$ be that 
interval in $\mathcal{P}_n$ closest to $c_1$. According to Lemma~\ref{OldLemma6}, if there are bridge intervals with righthand endpoint in $R$, the one closest to $c_1$ will be the bridge interval chosen by the algorithm and, moreover, will be a maximal bridge interval.
\par
We next prove \emph{all} component intervals of $Z_F^c$ (in $[a,c_1)$) come out of the algorithm. Assume, if possible, $M = (m_1, m_2)$ is a component not obtained by the algorithm.  Let $S = [s_1, s_2]$ be that member of $\mathcal{P}$ such that  $m_2 \in S$. 
\par
Now, either $S$ was chosen as an $R$ in some iteration or it was not. If it was chosen and $M$ is not the bridge interval with righthand endpoint in $S$ closest to $c_1$, then another bridge interval, $N = (n_1, n_2)$, is; in particular, $M$ and $N$ satisfy the hypotheses of Lemma \ref{OldLemma6}, with $m_2 < n_2$. We conclude $M \subset N$, which contradicts the maximality of $M$. 
\par
 Finally, suppose $S$ was not chosen. Then, there is a last iteration, say the $n$-th, such that
 $S \in \mathcal{P}_n$.  Let $T\in \mathcal{P}_n$ be the interval in $\mathcal{P}_n$ closest to $c_2$.
 
If $T$ does not contain the righthand endpoint of a bridge interval, $S$, will be chosen in the next iteration, which can't be. So, let $N = (n_1, n_2)$ be a bridge interval, indeed a component interval of $Z_F^c$, having $n_2 \in T$. Now, $n_1$ cannot be to the right of $S$ as that would entail $S \in \mathcal{P}_{n+1}$. Again, $n_1$ cannot lie to the left of $S$ nor can we have $n_1 < m_2$, since either would contradict the maximality of $M$. The only possibility left is $n_1 \in S$, $n_1 \geq m_2$. 
 \par 
 Should we have $n_1 > s_1$, $M$ would arise from $[s_1, n_1]$ in the next iteration. This  leaves the case $s_1 = m_2 = n_1$. All intervals in $\mathcal{P}_n$ contained in $[n_1, c_1] = [m_2, c_2]$ will be discarded at the end of the $n$-th step. But, according to Lemma \ref{OldLemma3}, there exists an interval in $\mathcal{P}_{n+1}$ with $m_2$ as its right hand endpoint, which interval will be the one in $\mathcal{P}_{n+1}$ closest to $c_1$. As $m_2$ belongs to that interval $M$ would come out of the $(n+1)$-th step of the algorithm contrary to our assumption.
 \end{proof}

\section{Implementation of the algorithm}\label{section5}

\par
In this section we discuss ways to make the algorithm more efficient. Suppose, then, that  $F$ is a differentiable piecewise polynomial and that we are searching for component intervals contained in $[a,c_1]$. In a given iteration we have chosen the interval $R=[r_1, r_2]$ furthest to the right in the current version of $\mathcal{P}$ and we are about to seek in it and, in an appropriate interval $L$ to the left, endpoints of a bridge interval. It turns out we needn't  do this for all $L$. 

 We developed have developed a few simple criteria to determine those $L$ which cannot contain the left endpoint of a bridge interval with right endpoint in $R$.

One natural test is to require of $L$ that $F'(L)  \cap F'(R) \neq \emptyset$.  

Lemma~\ref{Lemma7} below implies that there must be an intervening interval  in $P$ between $L$ and $R$  on which $F$ is convex (or linear).  We split the intervals in  $\mathcal{P}$ into groups such that intervals in the same group are not separated by any intervening convex or linear interval. Then, bridge intervals cannot have endpoints in intervals from the same group. Consequently $L$ is a viable candidate only if it belongs to a group other then $R$.

Moreover, for $L$ to be a viable candidate it must lie to the left of the set of points at which $F$ equals its maximum value on $[a,r_1]$.  This is a consequence of Lemma~\ref{Lemma8} as, it ensures that otherwise no bridge interval has endpoints in $L$ and $R$. 

Of course, there are more such criteria. We now state and proof the two Lemmas referred to above.

\begin{Lemma}\label{Lemma7}
Let $F$ be a differentiable piecewise polynomial function.  Every bridge interval has to contain an interval from $P$ on which $F$ is  not strictly concave. 
\end{Lemma} 

\begin{proof}
Suppose for contradiction that there is a bridge interval  $B = (b_1, b_2)$ such that  $F$ is strictly concave on   $(b_1, b_2)$. Condition \ref{bridge3} then yields that $F'(b_1) = F'(b_2)$. At the same time, strict concavity of $F$ yields that $F'$ is decreasing on $B$, which leads to contradiction. 
\end{proof}

\begin{Lemma}\label{Lemma8}
 Assume $F$ is a cubic spline, 
suppose $R=[r_1, r_2]   \subset [a, c_1]$ is an interval on which $F$ is strictly concave and increasing, with $m_2 \in R$ such that $\hat F (m_2) = F(m_2)$. Given $s < r_1$ satisfying $F(s) = \max \left\{ F(x) : x \in [a,r_1] \right\}$ and an $m_1 < r_1$ for which $M = (m_1, m_2)$ is a component interval of $Z_F^c$, one has $m_1 \in [a,s]$.     
\end{Lemma}

\begin{proof}
Assume, if possible, $m_1 \in (s,r_1]$. Then 
$\hat F (s) \geq F(s) \geq F(m_1) = \hat F (m_1)$, by hypothesis, and $\hat F (m_2) > \hat F(m_1)$, since $\hat F$ is increasing on $[a, c_1]$ according to Lemma~\ref{OldLemma1}. Hence 
\begin{eqnarray*}
\hat F(m_1)& \leq& \frac{m_2 - m_1} {m_2 -s } \hat F (s)  +  \frac{m_1 - s}{m_2 - s } \hat F (m_2) \\
&=& \hat F (s) + (m_1 -s ) \frac{\hat F(m_2) - \hat F (s)  }{m_2 - s },
\end{eqnarray*}
which contradicts the concavity of $\hat F$.    
\end{proof}

\section{Error Estimates}\label{section6}

Given an absolutely continuous function $G$ on a closed interval $I$ of finite length, we choose $F$ to be the clamped cubic spline interpolating $G$ at the points of a partition $\varrho$ of $I$. This permits us to take advantage of the following special case of optimal error bounds for cubic spline interpolation obtained by Charles A. Hall and W. Weston Meyer in \cite{HalMey1976}.

\begin{Proposition}\label{Proposition51}
Suppose $G \in \mathcal{C}^4 (I)$  and let $\varrho := [x_0, \dots ,x_{n+1}]$ be a partition of $I$. Denote by $F$ the clamped cubic spline interpolating $G$ at the nodes of $\varrho$. Then,
\[
\left| G' (x) - F' (x)\right|  \leq  \frac{1}{24} \left\| G^{(4)} \right\|_\infty \left\| \varrho \right\|^3, \;  x \in  I, 
\]
where $\| . \|_\infty$ denotes the usual supremum norm and
\[
\| \varrho \| := \sup \{ |x_k - x_{k-1} | : k = 1, \dots ,n\}.
\]
\end{Proposition}

To estimate the error involved in approximating the least concave majorant, we first consider the sensitivity of the level function to changes in the original function. We recall that the level function, $f^\circ$, of $f$ is given by $f^\circ = (\hat F)'$, where $F' = f$. 

\begin{Theorem}\label{theorem52}
Suppose $F$ and $G$ are absolutely continuous functions defined on a finite interval $I$. Then  $f = F'$, $g = G'$ and denote by $f^\circ$ and $g^\circ$ the level functions of $f$ and $g$ respectively. Then $\hat F$ and $\hat G$ are also absolutely continuous on $I$, and
\[
\left\|f^\circ  - g^\circ \right\|_\infty = \left\| (\hat F)' - (\hat G)' \right\|_\infty \leq \left\|  f - g \right\|_\infty.
\]
Here $\hat F$ and $\hat G$ denote the least concave majorants of $F$ and $G$, respectively, and $f = F'$, $g = G'$, while $f^\circ = (\hat F)'$, and $g^\circ = (\hat G)'$.
\end{Theorem}

\begin{proof} 
Set
\[
Z_F = \left\{ x\in I : F(x)=\hat F(x)\right\},\;
Z_G =\left\{x\in I : G(x)=\hat G(x)\right\}
\]
and observe that $f^\circ  = f$ almost everywhere on $Z_F$ and $g^\circ = g$ almost everywhere on $Z_G$. By Lemma \ref{SeptemberLemma1}, $\hat F$ is continuous and is of constant slope on each component of the complement of $Z_F$ . It follows that $\hat F$ is absolutely continuous on $I$. Since $\hat G$ is continuous and is of constant slope on each component of the complement of $Z_G$, $\hat G$ is absolutely continuous on $I$ as well.
We consider several cases to establish that $\left| f^\circ (x)-g^\circ (x)\right| \leq \left\| f - g\right\|_\infty$ for almost every $x \in  I$.
\begin{itemize}
\item Case 1: $x\in Z_F$ and $x\in Z_G$. For almost every such $x$, 
\[
\left| f^\circ (x) - g^\circ (x)\right| = \left| f(x) - g(x)\right| \leq  \left\| f - g\right\|_\infty.
\]
\item Case 2: $x\in Z_G$ but $x\notin Z_F$. Then $x$ is in the interior of some component interval $[a,b]$ of $F$. By Lemma \ref{SeptemberLemma1}, $\hat F(a) = F (a)$ and $\hat F(b) = F (b)$.
Since $\hat F$ has constant slope on $[a, b]$, 
\[
\int^x_a f =F(x)-F(a) \leq \hat F (x)-\hat F(a) = (x-a)f^\circ (x).
\]
and 
\[
\int^b_x f =F(b)-F(x) \geq \hat F(b)- \hat F(x)=(b-x)f^\circ (x).
\]
Also, since $\hat G(x) = G(x)$ and $g^\circ$ is non-increasing,
\[
\int_a^x g=G(x)-G(a) \geq \hat G (x) - \hat G (a) = \int_a^x g^\circ \geq (x-a) g^\circ (x).
\]
and
\[
\int_x^b g = G(b) - G(x) \leq \hat G(b) - \hat G (x) = \int_x^b g^\circ \leq (b-x) g^\circ (x).
\]
Combining these four inequalities, we obtain, 
\begin{eqnarray*}
- \left\| f - g\right\|_\infty &\leq& \frac{1}{ x - a} \int_a^x (f - g) \leq f^\circ (x) - g^\circ (x) \\
&\leq& \frac{1}{b - x} \int_x^b (f -g) \leq \left\| f - g\right\|_\infty .
\end{eqnarray*}
Thus, $\left| f^\circ (x) - g^\circ (x)\right| \leq  \left\| f - g\right\|_\infty$.
\item Case 3 : $x\in Z_F$  but $x\notin Z_G$. Just reverse the roles of $F$ and $G$ in Case 2.
\item Case 4: $x \notin Z_F$ and $x \notin Z_G$. Suppose without loss of generality that $g^\circ (x) \leq 
f^\circ (x)$. Let $a$ be the left-hand endpoint of the component interval of $G$ containing $x$, and let $b$ be the right-hand endpoint of the component interval of $F$ containing $x$. By Lemma~\ref{SeptemberLemma1}, $\hat G(a) = G(a)$ and $\hat  F(b) = F(b)$. Since $g^\circ $ is constant on $(a, x)$ and non-increasing on $(x, b)$ we have
\[
(b-a)g^\circ (x)\geq \int_a^b g^\circ  = \hat G(b)- \hat G(a) \geq G(b) -G(a)= \int_a^b g.
\]
Since $f^\circ$ is non-increasing on $(a,x)$ and constant on $(x,b)$, we have
\[
(b-a)f^\circ (x)\leq \int_a^b f^\circ  = \hat F(b) - \hat F(a) \leq F(b) - F(a)= \int_a^b f.
\]
Combining these, we have 
\[
f^\circ (x) - g^\circ (x) \leq \frac{1}{b-a} \int_a^b (f-g) \leq \left\| f - g \right\|_\infty. 
\]
\end{itemize}
This completes the proof.
\end{proof}

The last result can be combined with Proposition~\ref{Proposition51} to give the desired error
estimates.
\begin{Theorem}\label{Theorem53}
 Let $\varrho$ be a partition of the interval $[a,b]$ and suppose $G \in \mathcal{C}^4([a,b])$. Let $F$ be the clamped cubic spline interpolating $G$ on $\varrho$. Then 
 \[
\left\| f^\circ - g^\circ \right\|_\infty \leq \left\| f - g\right\|_\infty \leq \frac{1}{24} \left\| G^{(4)}\right\|_\infty  \left\| \varrho \right\|^3
 \]
and for each $x \in [a, b]$,
 \[
\left| \hat F (x) - \hat G(x)\right| \leq  \frac{\min \{ x-a,b-x \}}{24} \left\| G^{(4)}\right\|_\infty \left\| \varrho \right\|^3.
\]
Here $\hat F$ and $\hat G$ denote the least concave majorants of $F$ and $G$, respectively, and $f = F'$, $g = G'$; $f^\circ  = (\hat F)'$, and $g^\circ  = (\hat G)'$.
 \end{Theorem}
 \begin{proof}
  The first inequality is just Theorem~\ref{theorem52} together with the result from \cite{HalMey1976}. For the second, observe that by Lemma~\ref{SeptemberLemma1}, $\hat F(a) = F(a)$ and $\hat G(a) = G(a)$, and since $a$ is in the partition $\varrho$, $G(a) = F(a)$. Thus, $\hat F(a) = G(a)$. Since both $\hat F$ and $\hat G$ are concave and hence absolutely continuous,
 \[
 \left| \hat F (x) - \hat G (x) \right| = \left| \int_a^x  f^\circ  (x) -  g^\circ  (x)  \right|  \leq \int_a^x \left\| f^\circ  - g^\circ  \right\|_\infty \leq \frac{x - a}{24} \left\| G^{(4)} \right\|_\infty \left\| \varrho \right\|^3.
 \]
A similar argument, using integration on $[x, B]$, shows that
 \[
\left| \hat F(x) -  \hat G(x)\right|  \leq \frac{ b - x}{24} \left\| G^{(4)}\right\|_\infty  \left\| \varrho \right\|^3
 \]
and completes the proof.
\end{proof}

\section{Examples}\label{section7}

We present here two examples involving our algorithm. 
\subsection*{Example 1.}
With our first example we illustrate the flow of the algorithm. Let $s$ be the continuously differentiable, piecewise cubic function defined on $[0,10]$, by 
\[
s(x) = s_n(x) \mbox{ on } [n-1,n],\; n = 1,2, \dots 10,
\]    
where
\begin{eqnarray*}
s_1 (x) = -1.1 x^3 + 1.1 x^2 + x + 1, && s_2 (x) = 1.3 x^3 - 5.3 x^2 + 6.6 x - 0.6, \\
s_3 (x) = -0.9 x^3 + 1.1 x^2 +  x + 1, && s_4 (x) = -1.5 x^3 +16 x^2 -56 x +67, \\
s_5 (x) =  3,  && s_6 (x) = 0.5 x^3 - 8.75 x^2 + 50 x - 90.75, \\
s_7 (x) = 2 + (x - 6.5)^2, &&  s_8 (x) = 1.5 x^3 - 33.25 x^2 + 246 x +605, \\
s_9 (x) =  x^3 - 25.5 x^2 + 216 x - 605, && s_{10} (x) = 0.6 x^3 - 16.6 x^2 + 153 x - 467.3.
\end{eqnarray*}
The graph of $s$ is given in Figure~\ref{figure1} below.

\begin{figure}
\includegraphics[scale = 0.425]{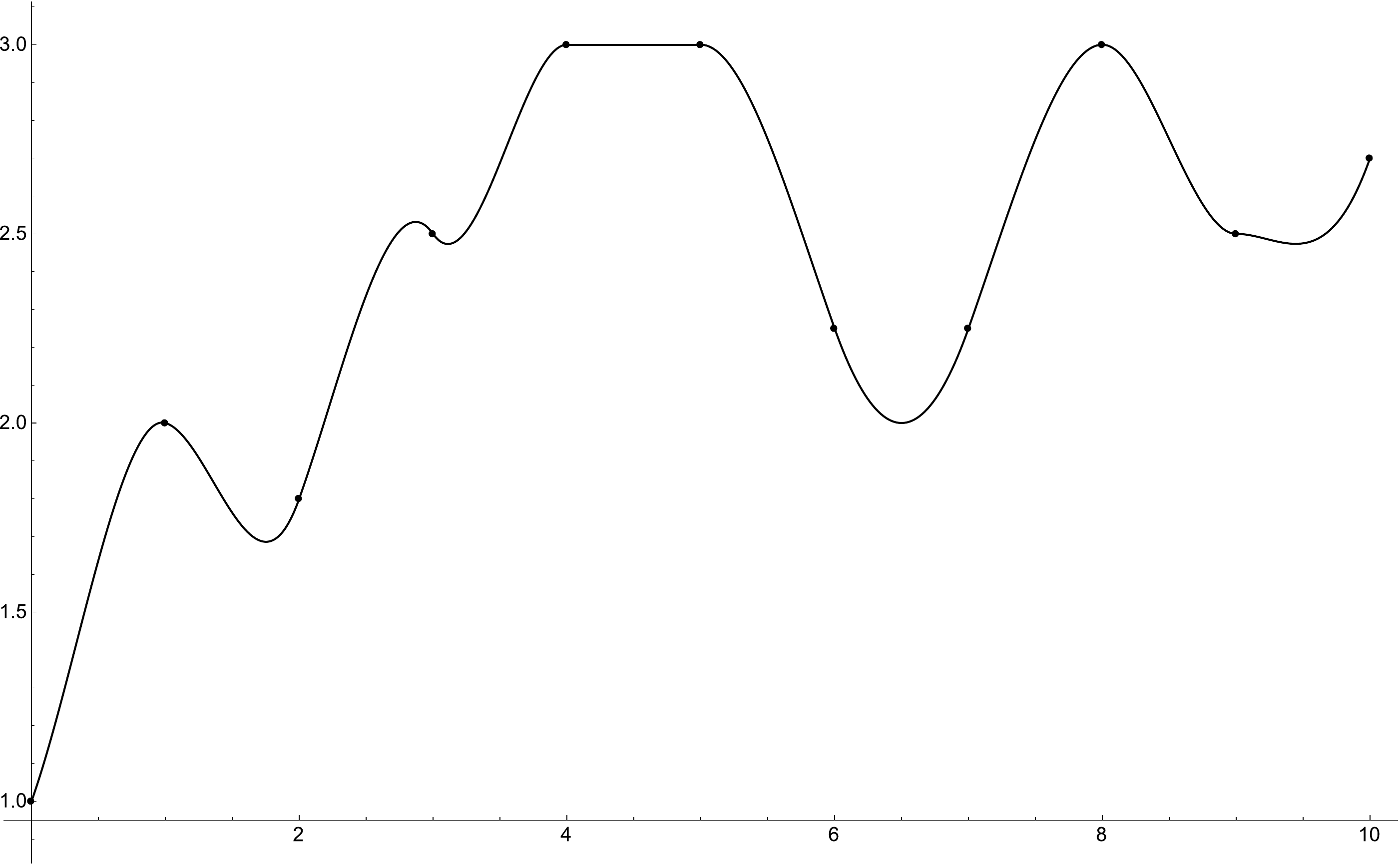}
\put(-100, 108){$s$}
\caption{Graph of $s$ with marked points where prescribed polynomials change.}
\label{figure1}
\end{figure}

To begin, $s$ attains its maximum value of $3$ on $D = [4,5]\cup \left\{ 8 \right\}$. So, $\hat{s} (x) = 3$ on $C = [4,8]$.

Since $s < 3$ on $(5,8)$ it will be a component interval. We next seek the component intervals in $[0,4]$. 
By adding to the partition those points in $[0,4]$ for which $s'$ or $s''$ changes sign we get a refined partition where, on each subinterval, $s$ is monotone and either strictly convex or strictly concave. The first derivative of $s$ changes sign at $0.97687$, $1.75204$, $2.8701$ and $3.\bar{1}$. The second derivative changes sign at $0.\bar{3}$, $1.35897$, $2.\bar{2}$ and $3.\bar{5}$. We are interested in subintervals of $[0,4]$ where $s$ is strictly concave and increasing. These are  $I_1 = [0.\bar{3}, 0.97687]$, $I_2 = [2.\bar{2}, 2.87011]$ and $I_3 = [3.\bar{5}, 4]$.
Thus, $\mathcal{P}_0 = \left\{ I_1 , I_2, I_3 \right\}$. Clearly, $I_3$ is the interval in $\mathcal{P}_0$ furthest to the right.   

There are no bridge intervals with left endpoint $0$ and right endpoint in $I_3$.

Indeed, there \emph{are} two candidate intervals of form $[a,r]$, $r \in I_3$, such that  
\begin{equation}\label{condLeft}
s'(r) = \frac{s(r) - s(0)}{r} = \frac{s(r) - 1}{r},
\end{equation}
but, for neither candidate does one have (3), that is,
\[
s(x) <  x \left[ \frac{s(r) - 1 }{r}\right], \; x \in (0,r).
\]
This can be seen in Figure~\ref{SimpleLeftEndPoint}. 
\begin{figure}
\includegraphics[scale = 0.5]{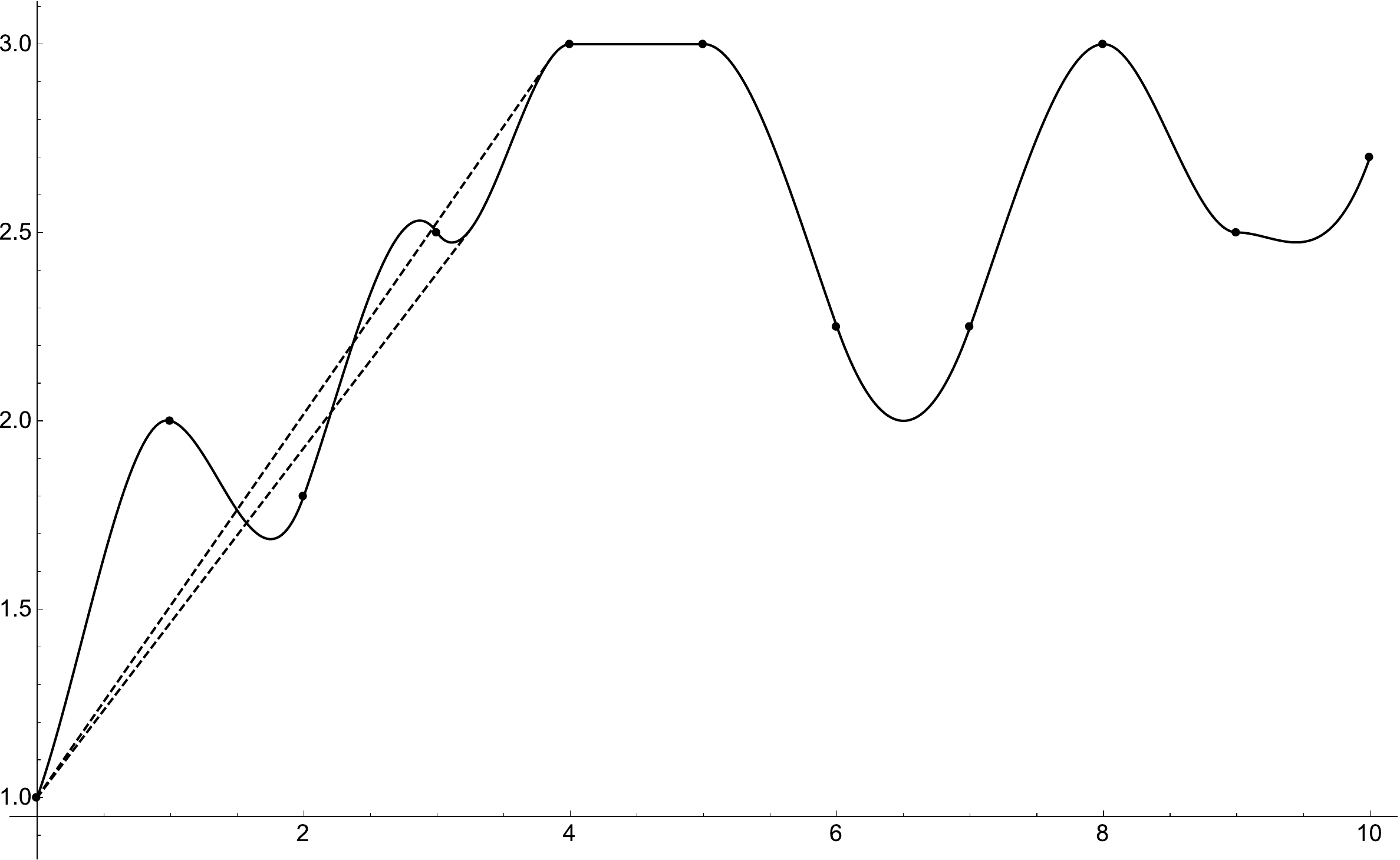}
\put(-100, 120){$s$}
\caption{The two intervals with left-hand endpoint being $0$ which satisfy the first condition (\ref{condLeft}) are $[0, 3.24826]$ and $[0, 3.84606]$. But both cannot meet the second condition from definition of bridge interval. }
\label{SimpleLeftEndPoint}
\end{figure}
Again, there are two intervals with right endpoint in $I_3$ and left endpoint in $I_1$ for which  (1) and (2) holds. These are 
\[
I_{1,1} = (0.89359, 3.90772) \mbox{ and } I_{1,2} = (0.92390, 3.16878).
\]    
However, only on $I_{1,1}$ is (3) satisfied. The situation is depicted in Figure~\ref{figure3}.   
\begin{figure}
\includegraphics[scale = 0.455]{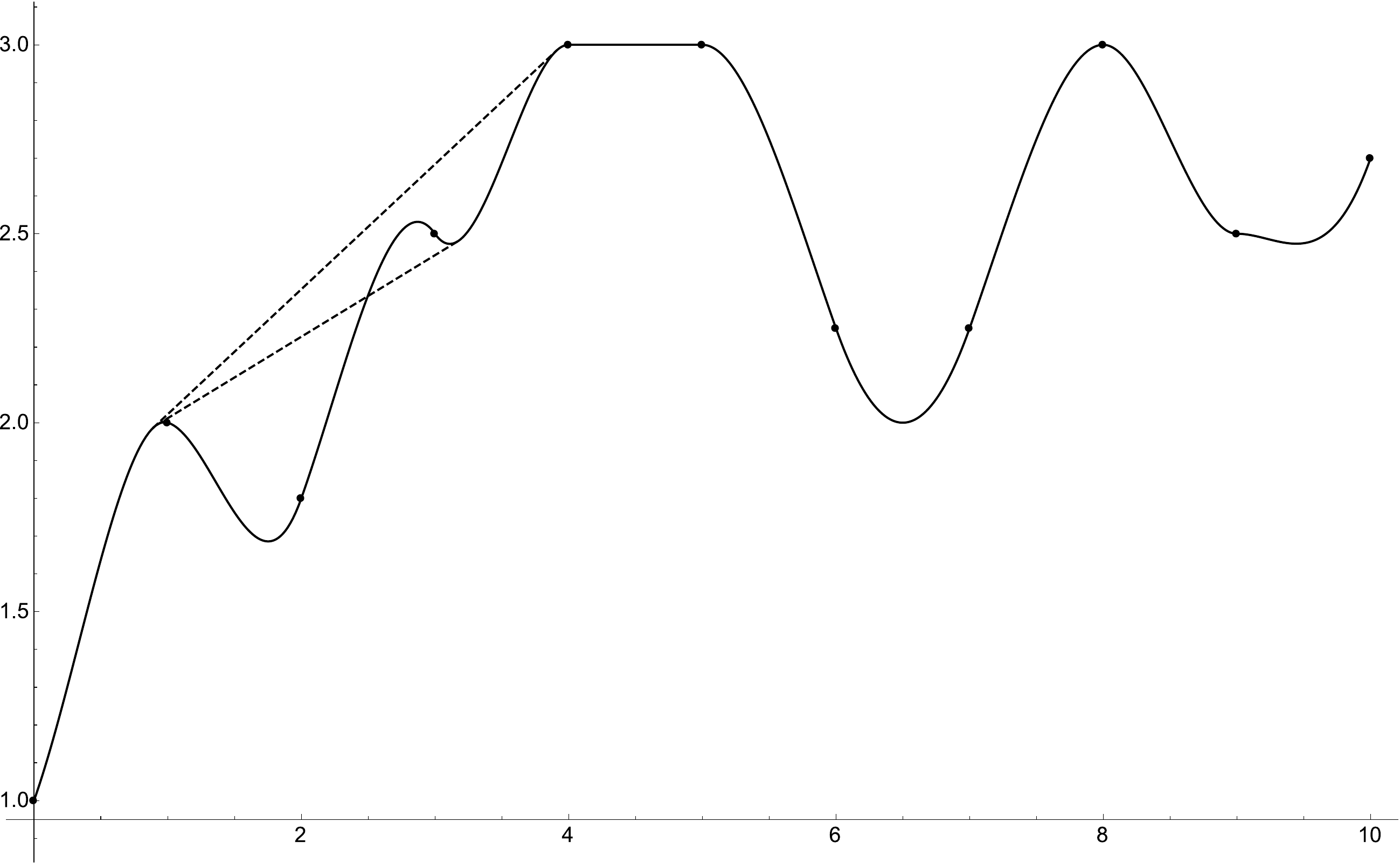}
\put(-100, 120){$s$}
\caption{This figure pictures the bridge interval joining intervals $I_1$ and $I_3$ and the other candidate. }
\label{figure3}
\end{figure}

Since no interval with left endpoint in $I_2$ can have smaller left endpoint than  left endpoint of $I_{1,1}$, the interval $I_{1,1}$ is the desired component interval.
This completes the first iteration of our algorithm.

To form $\mathcal{P}_1$ for the second iteration we, of course, discard $I_3$. We also discard $I_2$, since it is contained int $I_{1,1}$. This leaves in $\mathcal{P}_1$ only interval $I_{1}'$, as $(0.\bar{3}, 0.89359) = I_1 \setminus I_{1,1}$.

There is one bridge interval with right endpoint in $I_{1}'$ and left endpoint $0$. It is $(0,0.5)$, therefore $(0,0.5)$ is a component interval. We have thus found all component intervals in $[0,4]$.

We now seek component intervals contained in $[8,10]$. To begin we must ad to the partition points 8,9,10 the critical point $8.5$ and the inflection points $9.\bar{2}$ and $9.\bar{4}$. It is then found that the intervals on which $s$ is strictly concave  and increasing are $J_1 = [8,8.5]$ and $J_2 = [9,9.\bar{2}]$.

The interval $[8.05353,10]$ is a bridge interval with left endpoint in $J_1$ and right endpoint $10$.

The unique component interval in $[8,10]$. See Figure~\ref{SimpleRightEndPoint}.
\begin{figure}
\includegraphics[scale = 0.47]{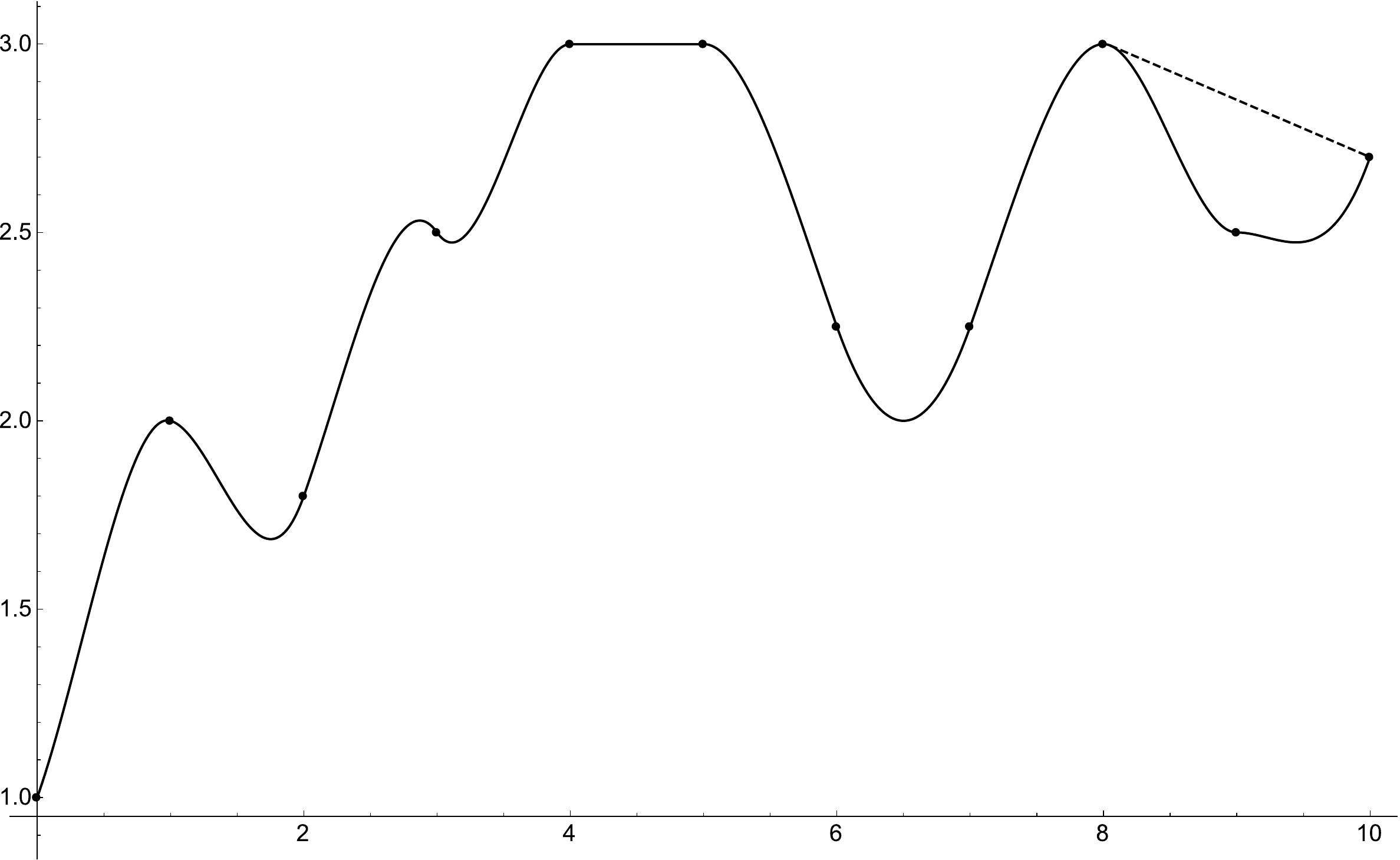}
\put(-95, 120){$s$}
\caption{This figure shows the component interval $(8.05353,10)$. }
\label{SimpleRightEndPoint}
\end{figure}

The graph of $\hat s $ appears in Figure~\ref{figure6}.
\begin{figure}
\includegraphics[scale = 0.48]{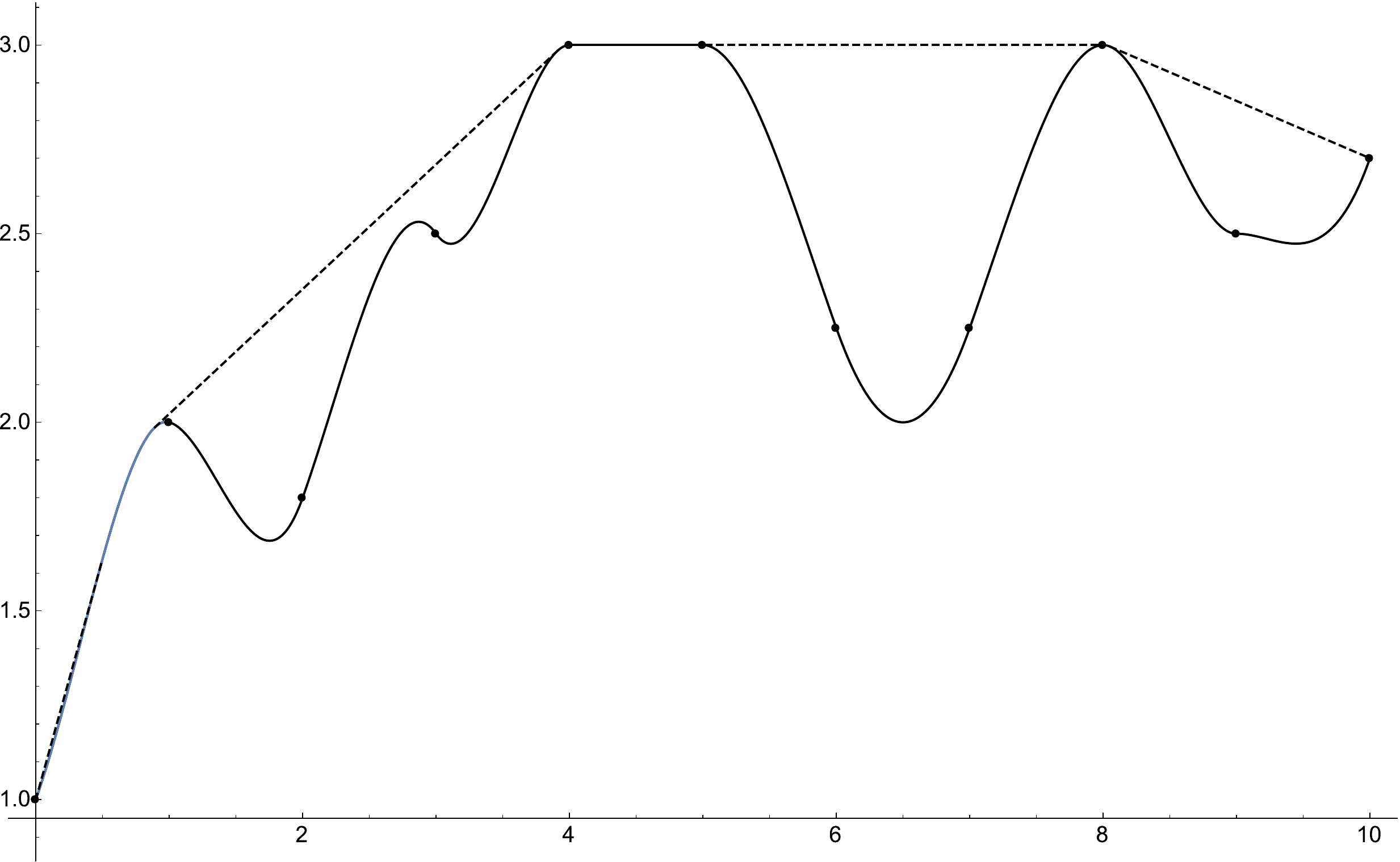}
\put(-130, 100){$s$}
\put(-140, 190){$\hat s$}
\caption{The least concave majorant of $s$ is linear interpolation of $s$ from end-points of a component interval and agrees with $s$ elsewhere.}
\label{figure6}
\end{figure}

\subsection*{ Example 2.}

Consider the trimodal density function discussed in \cite{HarKerPicTsy1998}, namely,
\[
f(x) = 0.5 \phi(x-3) + 3 \phi( 10(x - 3.8)) + 2 \phi(10(x-4.2)),
\]
in which 
\[
\phi(x) = \frac{1}{\sqrt{2\pi}} e^{-\frac{x^2}{2} }.
\]
We wish to approximate the least concave majorant of $F(x) = \int_0^x f(y) \,dy$ on $[0,6]$. Now, $\left\|F^{(4)} \right\|_\infty \leq 700$, so to ensure that the clamped cubic spline $S_F$ approximating $F$ on $[0,6]$  satisfies $| f^\circ (x) - (S'_F)^\circ (x) | \leq .001$ on $[0,6]$, we solve the equation $\frac{700}{24}\left\| \varrho \right\|^3 = .001$ to obtain $\left\| \varrho \right\| = .03249$. Dividing $[0,6]$ into $85 > \frac{6}{.03249}$ equal subintervals, we apply the algorithm to identify the component intervals of $Z_{S_f}^C$. The approximation $\int_0^x (\hat S'_f)^\circ$  to $\hat F (x)$ is accurate to within $.003$.

Figure \ref{figure62} shows the graph of $F(y)$ and the approximation to its least concave majorant, $\hat S_F$.

\begin{figure}\label{figure62}
\includegraphics[scale = 0.485]{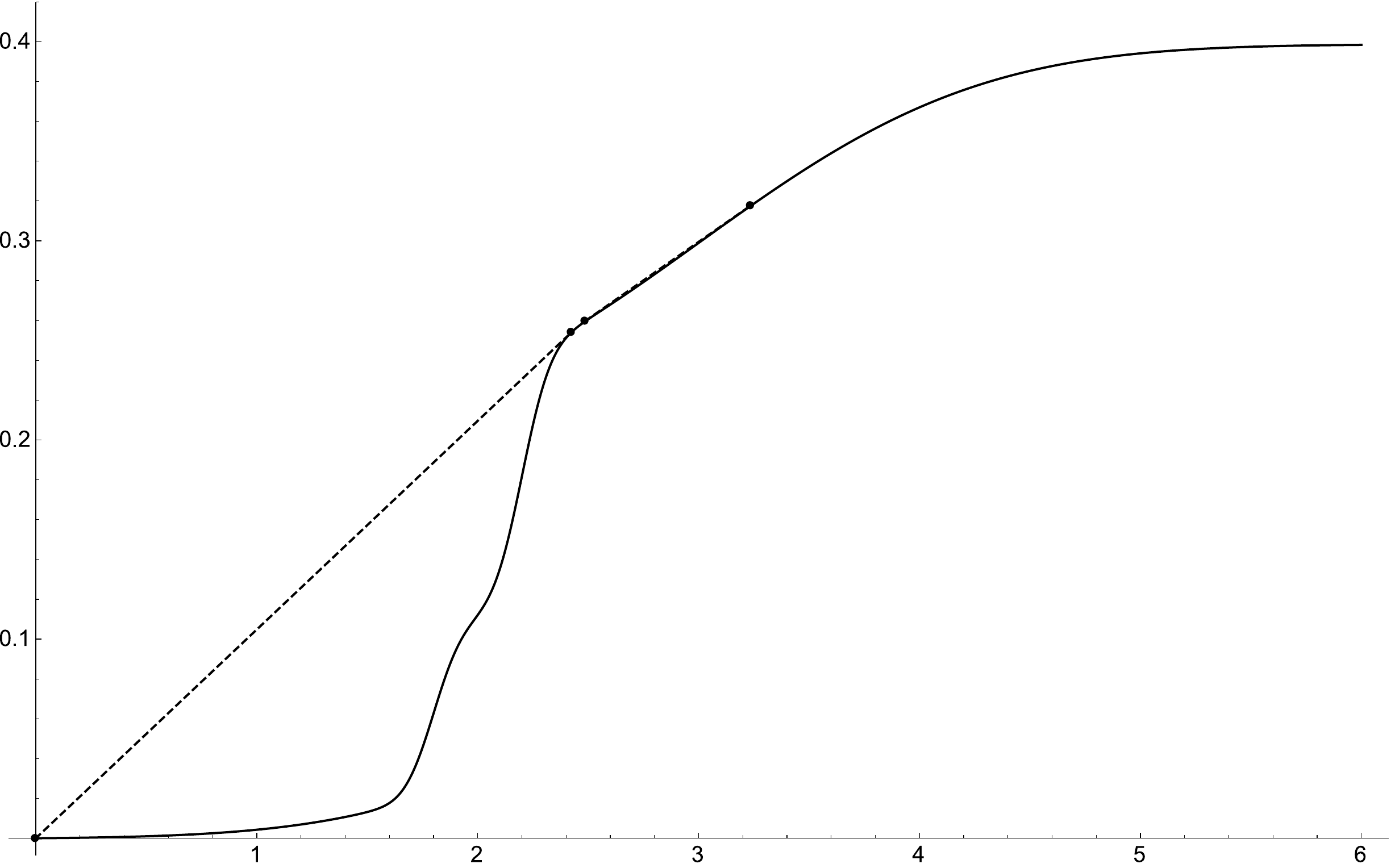}
\put(-50,193){$F$}
\put(-260 ,70){$\hat F$}
\caption{The trimodal density function $F$ with its least concave majorant $\hat F$.  The bridge intervals are $(0, 2.42575)$ and $(2.48781,3.23693) $.}
\end{figure}

{\small
{\em Authors' addresses}: 
{\em Martin Franc\accent23u}, 
Faculty of Mathematics and Physics, Charles~University, 
Prague, Czech Republic,
 e-mail:~\texttt{martinfrancu@gmail.com};  
{\em Ron Kerman},
Department of Mathematics and Statistics, Brock~University,
St. Catharines, Canada,
 e-mail:~\texttt{rkerman@brocku.ca};
{\em Gord Sinnamon},
Department of Mathematics, University of~Western Ontario,
London, Canada,
e-mail:~\texttt{sinnamon@\allowbreak uwo.ca};
}

\end{document}